\documentclass[12pt]{amsart}
\usepackage{enumitem}
\usepackage[all]{xy}
\usepackage[margin=1.00in, marginparwidth=1.5cm, marginparsep=0.5cm]{geometry}

\usepackage{mathtools}
\usepackage{bbold}

\usepackage{tikz}
\usetikzlibrary{shapes.geometric}
\usetikzlibrary{arrows}

\usepackage{verbatim}
\usepackage{amssymb}
\usepackage{array}

\usepackage{mathrsfs}
\usepackage{physics}
\usepackage{amsthm}
\usepackage{amsmath,amsfonts,graphicx,mathdots,cite,wrapfig}
\usepackage[colorlinks=true,linkcolor=blue,citecolor=blue,urlcolor=blue,pagebackref,hypertexnames=false, breaklinks]{hyperref}
\usepackage{cite}
\usepackage{mwe}
\usepackage{caption}
\usepackage{marginnote}
\usepackage{tcolorbox}

\newtheorem{theorem}{Theorem}[section]

\newtheorem{proposition}[theorem]{Proposition}
\newtheorem{corollary}[theorem]{Corollary}

\theoremstyle{definition}
\newtheorem{definition}[theorem]{Definition}

\theoremstyle{remark}

\numberwithin{equation}{section}

\begin{document}

\title{Orthogonal Polynomials on Bubble-Diamond Fractals}

\pagestyle{myheadings} \thispagestyle{plain}
\markboth{E. P. AXINN, C. OSBORNE, K. A. OKOUDJOU, O. RIGATTI, AND H. SHI}{ORTHOGONAL POLYNOMIALS ON BUBBLE-DIAMOND FRACTALS}

\author{Elena Axinn}
\address{Department of Mathematics, University of Michigan, Ann Arbor, MI 48109}
\email{eaxinn@umich.edu}

\author{Calvin Osborne}
\address{Department of Mathematics, Harvard University, Cambridge, MA 02138}
\email{cosborne@alumni.harvard.edu}

\author{Kasso A.~Okoudjou}
\address{Department of Mathematics, Tufts University, Medford, MA 02155}
\email{kasso.okoudjou@tufts.edu}

\author{Olivia Rigatti}
\address{Department of Mathematics, North Carolina State University in Raleigh, NC 27695}
\email{origatt@ncsu.edu}

\author{Helen Shi}
\address{Department of Mathematics, Tufts University, Medford, MA 02155}
\email{yues.0510@gmail.com}

\date{\today}


\begin{abstract}
We develop a theory of polynomials and, in particular, an analog of the theory of Legendre orthogonal polynomials on the bubble-diamond fractals, a class of fractal sets that can be viewed as the completion of a limit of a sequence of finite graph approximations. In this setting, a polynomial of degree $j$ can be viewed as a multiharmonic function, a solution of the equation $\Delta^{j+1}u=0$. We prove that the sequence of orthogonal polynomials we construct obeys a three-term recursion formula. 
\end{abstract}

\maketitle
\pagestyle{myheadings} \thispagestyle{plain}
\markboth{E. A. AXINN, C. OSBORNE, K. A. OKOUDJOU, O. RIGATTI, AND H. SHI}{ORTHOGONAL POLYNOMIALS ON BUBBLE-DIAMOND FRACTALS}

\section{Introduction}

During the last two decades, 
a theory of calculus on fractal sets such as the Sierpinski gasket ($SG$) has been developed. It is based on the spectral analysis of the fractal Laplacian \cite{CalcSGII,CalcSGI,splinesBob}.
In this context, a polynomial of degree $j$ on $SG$ is a solution of the equation $\Delta^{j+1}u=0$, where $\Delta$ denoted the Laplacian on $SG$ \cite{KigamiBook2001,BobBook2006}. Although most aspects of the theory of polynomials in this setting parallel their counterparts on the unit interval, several striking differences exist. In particular, there is no analog of the Weierstrass Theorem on $SG$; that is, the set of polynomials on $SG$ is not complete on $L^2(SG)$ \cite[Theorem 4.3.6]{KigamiBook2001}. Furthermore, the space of polynomials on $SG$ is not an algebra under pointwise multiplication. However, a theory of orthogonal polynomials on $SG$ was initiated in \cite{OST} and resulted in an analog of Legendre orthogonal polynomials on $[-1,1]$. 

This paper studies the bubble-diamond fractals, a class of fractals defined by a branching parameter $b \in \mathbb{N}$. One significant interest in exploring this class of fractals is that they present different geometrical properties, including a wide range of Hausdorff and spectral dimensions. In general, diamond-type self-similar graphs have provided an essential collection of structures with interesting physical and mathematical properties and a broad variety of geometries  \cite{MalozemovTeplyaev1995,Akkermans2009ComplexDim, HamblyKumagai2010}. The structure of these fractals is such that they combine spectral properties of Dyson hierarchical models and transport properties of one-dimensional chains. In what follows, we will denote by $K_b$ the bubble-diamond fractal with branching parameter $b\in \mathbb{N}$, see Definition~\ref{def:bdf}.

In this paper, we develop a theory of polynomials and orthogonal polynomials theory on this class of bubble-diamond fractals. First, in Section~\ref{sec:defbdf}, we define the bubble-diamond fractal $K_b$ as a particular completion of a limit of self-similar bubble-diamond graphs introduced in \cite{BubbleDiamond2023}. Subsequently, we introduce the main analytic tools on $K_b$: the Laplacian and the Green operator. Next, in Section~\ref{sec:Poly}, we define and investigate the fundamental properties of polynomials on the bubble-diamond fractals following the approach developed in  \cite{splinesBob}. Finally, in Section~\ref{sec:OP}, we introduce a class of monomials and construct analogues of the Legendre orthogonal polynomials on $K_b$, providing asymptotic results on relevant coefficients and defining a three-term recursion relation for the Legendre orthogonal polynomials.

\section{Bubble-Diamond Fractals}\label{sec:defbdf}

In this section, we define the bubble-diamond fractals as the completion of the limit of finite graph approximations. To this end, we will first define these finite graph approximations, then construct a suitable metric on the union of these finite graph approximations, and finally define the bubble-diamond fractals. Throughout the process, we will develop several analytic tools on $K_b$ that will be essential in creating a theory of polynomials on the bubble-diamond fractal.

\subsection{Analytic Tools on Finite Bubble-Diamond Graphs} First, we will construct a sequence of finite bubble-diamond graphs $G_\ell$ that approximate $K_b$. In the process, we will develop several analytic tools, including the graph Laplacian $\Delta_\ell$ and Green's function $G_\ell$, on these finite graph approximations that we will extend to $K_b$ in the following subsection. Note below, as will occur often throughout this paper, that we drop the dependence on the branching parameter $b \in \mathbb{N}$ to simplify our notation.

\begin{figure}[b!]
        \centering
        \includegraphics[width=1.0\linewidth]{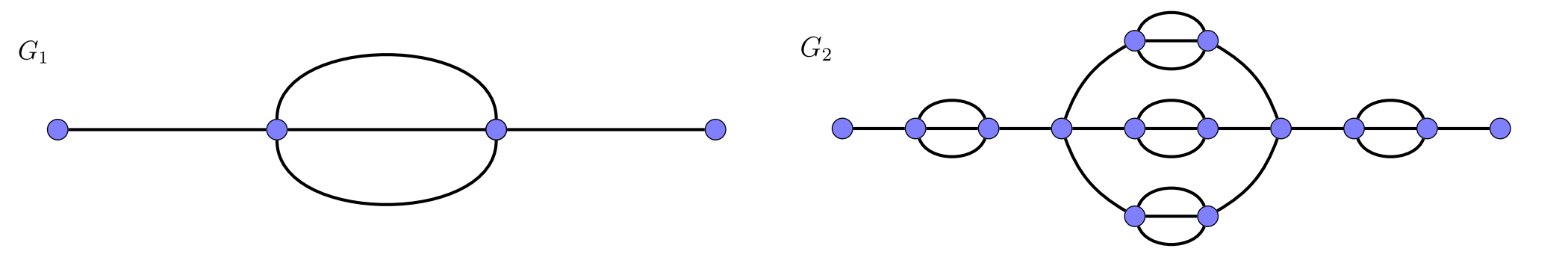}
    	\caption{Bubble-diamond graphs $G_1$ and $G_2$ for $b=3$.}
    	\label{fig:bubblediamondgraphb3}
\end{figure}

\begin{definition}\label{def:BubbleDiamondGraphs}
For an integer $b\geq 1$, the  \textit{bubble-diamond graphs with branching parameter $b$} are an inductively constructed sequence of graphs $G_{\ell}=(V_{\ell},E_{\ell})$ for $\ell \in \mathbb{N}$, where $V_{\ell}$ and $E_\ell$ are respectively the sets of vertices and edges of $G_\ell$. 

First, we will let $G_{0}=(V_0,E_0)$ be the graph with two vertices $V_0=\{q_1,q_2\}$ joined by an edge $(q_1,q_2) \in E_0$. Then, at level $\ell \in \mathbb{N}$, we construct $G_{\ell}$ by modifying each edge from  $G_{\ell-1}$ by introducing two new vertices, each of which is joined to one of the two original vertices by an edge and which are joined to each another by $b$ distinct edges. Note in particular that $V_0 \subset V_1 \subset \dots$ for each subsequent level of bubble-diamond graphs.

For a particular instance of bubble-diamond graphs, see $G_1$ and $G_2$ for $b = 3$ in Figure~\ref{fig:bubblediamondgraphb3}.
\end{definition}

There is an equivalent construction of the vertices $V_\ell$ of the bubble-diamond graphs through a set of mappings that will be important in constructing polynomials on the bubble-diamond fractals. To see this, we note that we can realize the general structure of $G_{\ell + 1}$ by gluing together $b + 2$ copies of $G_\ell$ as in Figure~\ref{fig:bubblediamondgraphgeneral}. This gives us a series of maps $F^\ell_i : V_\ell \to V_{\ell + 1}$ for $i \in \{1, 2, \dots, b + 2\}$ given by
\begin{equation}
    F_i^\ell(p) = \text{the corresponding point to }p\text{ in }G_\ell^{\text{copy }i}.
\end{equation}
One important observation is that $F_i^\ell(p) = F_i^{\ell + n}(p)$ for all $n > 0$ by considering $V_\ell \subset V_{\ell + n}$, and so we will drop the dependence on $\ell$ in the notation $F_i : V_\ell \to V_{\ell + 1}$. With this map defined, we can construct 
\begin{equation}
    V_{\ell + 1} = \bigcup_{i \in \{1, 2, \dots, b + 2\}}F_i(V_{\ell}),
\end{equation}
which is equivalent to the direct geometric construction of $V_{\ell + 1}$ above.

\begin{figure}[b!]
        \centering
        \includegraphics[width=0.7\linewidth]{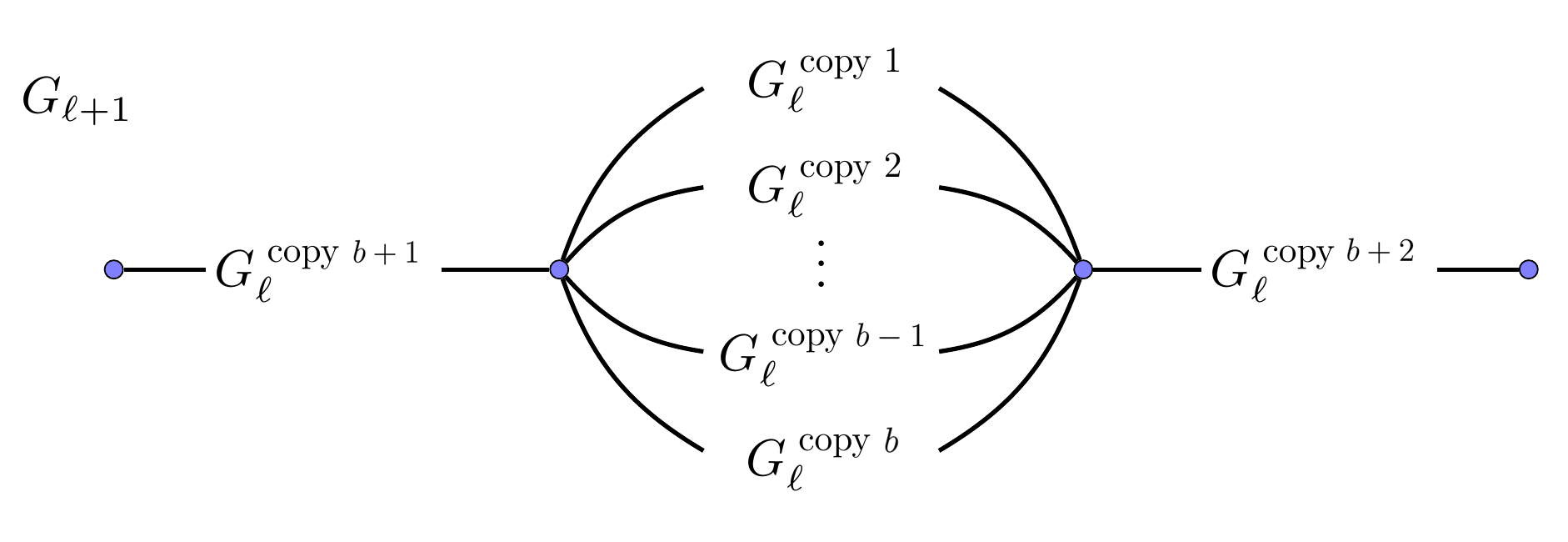}
    \caption{Constructing $G_{\ell+1}$ by gluing together $b+2$ copies of $G_{\ell}$.}
    \label{fig:bubblediamondgraphgeneral}
\end{figure}

We will now develop a number of analytic tools on the finite graph approximations of $K_b$. First, we define the \textit{graph Laplacian} of a function $f$ defined on the vertices $V_\ell$ of a graph approximation $G_\ell$ given by
\begin{equation}\label{eq:graphlaplacianfinite}
    \Delta_{\ell} f(p)= \Big( \frac{1}{\deg(p)} \sum_{p \sim q} f(q) \Big) -f(p),
\end{equation}
where the sum is taken over every edge $\{p, q\} \in E_\ell$. The finite graph Laplacian has an associated energy functional
\begin{equation}\label{eq:finitegraphenergy}
    \mathcal{E}_\ell(f) = \left(\dfrac{b}{2b + 1}\right)^{-\ell}\sum_{p \sim q}\abs{f(p) - f(q)}^2
\end{equation}
known as the \textit{graph energy}. The constant $r = b/(2b + 1)$ in the above formula is known as the \textit{renormalization constant} for the bubble-diamond graphs, which is defined so that $\mathcal{E}_\ell$ satisfies the following energy-minimizing property.

\begin{proposition}\label{prop:renormalizationprop}

For any function $f$ defined on $V_\ell$ and any $\ell' \geq \ell$, we have that $\mathcal{E}_\ell(f) \leq \mathcal{E}_{\ell'}(\widetilde{f})$ for any extension $\widetilde{f}$ of $f$ to $V_{\ell'}$, i.e., for any function $\widetilde{f}$ on $V_{\ell'}$ that satisfies $\widetilde{f}|_{V_\ell} = f$.
    
\end{proposition}

\begin{proof}

From the second construction of $G_\ell$ above, it is clear that it suffices to show that this property holds from $\ell = 0$ to $\ell' = 1$. Letting $V_1 = \{q_1, p_1, p_2, q_2\}$ be the vertices of $G_1$, we compute for any function $f$ on $V_0$ and extension $\widetilde{f}$ to $V_1$ that
\begin{align*}
    \mathcal{E}_1(\widetilde{f}) &= r^{-1}\left((f(q_1) - \widetilde{f}(p_1))^2 + (f(q_2) - \widetilde{f}(p_2))^2 + b(\widetilde{f}(p_1) - \widetilde{f}(p_2))^2\right).
\end{align*}
By taking partial derivatives, we find that $\mathcal{E}_1(\widetilde{f})$ is minimized when
\begin{equation*}
    \begin{cases}
        \widetilde{f}(p_1) = \left(\dfrac{b+1}{2b+1}\right) f(q_1) + \left(\dfrac{b}{2b+1}\right) f(q_2), \\[3ex]
        \widetilde{f}(p_2) = \left(\dfrac{b}{2b+1}\right) f(q_1) + \left(\dfrac{b+1}{2b+1}\right) f(q_2),
    \end{cases} \label{eq:harmonic_ex}
\end{equation*}
which gives us exactly the minimum graph energy $\mathcal{E}_1(\widetilde{f}) = \mathcal{E}_0(f)$ as is desired.
\end{proof}

We further note that the graph energy induces a bilinear form through the polarization identity
\begin{equation} \label{eq:polarizationidentity}
    \mathcal{E}_\ell(f, g) = \dfrac{1}{4}\left(\mathcal{E}_\ell(f + g) - \mathcal{E}_\ell(f - g)\right).    
\end{equation}
To see how the graph Laplacian $\Delta_\ell$ and graph energy $\mathcal{E}_\ell$ on the finite bubble-diamond graphs are associated, we will turn to the Gauss-Green formula in the following subsection.

Lastly, we will construct a function $G_\ell$ on $V_\ell \times V_\ell$ that will serve as an analogue of \textit{Green's function} in classical calculus. To this end, we will define
\begin{equation}\label{eq:greensfunctionfinite}
    G_\ell(p,q) = \sum_{m=0}^\ell \sum_{x,y \in V_{m+1} / V_m} g(x,y)\psi_x^{(m+1)}(p) \psi_{y}^{(m+1)}(q),
\end{equation}
where $g(x,x) = \alpha r^m$ and $g(x,y) = \beta r^m$ for $x \neq y$ with constants $\alpha = r^2(b+1)/b$ and $\beta = r^2$, and where $\psi_x^{(m)}$ is the \textit{piecewise harmonic function} on $V_m$ defined by $\psi_x^{(m)}(p)=\delta_{xp}$. Again, to see why $G_\ell$ is an analogue of Green's function on $V_\ell$, we will turn to solving the Poisson equation in the following subsection.

\subsection{Bubble-Diamond Fractals}

We will now construct the bubble-diamond fractals from the above finite graph approximations $G_\ell$. First, we will define the \textit{graph approximation} 
\begin{equation}
    V_* = \bigcup_{\ell \geq 0}V_\ell.
\end{equation}
We note that we can extend in a natural way the maps $F_i : V_\ell \to V_{\ell + 1}$ that define the finite bubble-diamond graphs to maps $F_i : V_* \to V_*$ on the graph approximation. We will now extend the graph energy as defined in~\eqref{eq:finitegraphenergy} to the graph approximation in the following way.

\begin{definition}

The \textit{graph energy} of a function $f$ defined on $V_*$ is defined to be
\begin{equation} \label{eq:graphenergy}
    \mathcal{E}(f) = \lim_{\ell \to \infty} \mathcal{E}_\ell(f|_{V_\ell}).
\end{equation}
From this, we will define the \textit{resistance metric} $R(p,q)$ between points $p,q \in V_*$ to be the minimum value of $R$ such that
\begin{equation} \label{eq:resistancemetric}
    \abs{f(p) - f(q)}^2 \leq R \mathcal{E}(f)
\end{equation}
for all $f \in \text{dom}\,\mathcal{E}$, i.e., for all functions $f$ on $V_*$ such that $\mathcal{E}(f) < \infty$.

As with above, we extend the graph energy on $V_*$ to a bilinear form $\mathcal{E}(f,g)$ through the polarization identity~\eqref{eq:polarizationidentity}. Both the graph energy and resistance metric play an important role for the graph approximation, as shown in the following result. 

\begin{proposition} 

The graph energy $\mathcal{E}$ is an inner product on $\text{dom}\,\mathcal{E} / (\text{constant functions})$. Furthermore, the resistance metric is a metric on $V_*$.

\end{proposition}

\begin{proof}

It is clear that $\mathcal{E}$ is symmetric, linear in both arguments, and satisfies $\mathcal{E}(f) \geq 0$, so it remains to show that $\mathcal{E}(f) = 0$ if and only if $f$ is constant. We recall by Proposition~\ref{prop:renormalizationprop} that $\mathcal{E}_\ell$ is nondecreasing as $\ell \to \infty$, and so we have that
\begin{equation*}
    \mathcal{E}(f) = 0 \Leftrightarrow \mathcal{E}_\ell(f) = 0\text{ for all }\ell \in \mathbb{N}.
\end{equation*}
It is clear from~\eqref{eq:finitegraphenergy} that $\mathcal{E}_\ell(f) = 0$ if and only if $f|_{V_\ell}$ is constant, and so $\mathcal{E}(f) = 0$ if and only if $f$ is constant on $V_*$ as is desired.

Next, it is clear that $R$ is symmetric and positive-definite, so it remains to show that $R$ satisfies the triangle inequality. To this end, we will consider distinct $p,q,z \in V_*$, and we note that there is some $\ell \in \mathbb{N}$ such that $p,q,z \in V_\ell$. It suffices to consider the case that $p \sim q$ and $q \sim z$ in $V_\ell$. We will show the result when $\ell = 1$, where we will let $V_1 = \{q_1, p_1, p_2, q_2\}$ be the vertices of $G_1$, and we note that the result generalizes for $\ell > 1$ by the energy-minimizing algorithm in Proposition~\ref{prop:renormalizationprop}. We will assume without a loss of generality that $(p, q, z) = (q_1, p_1, p_2)$. Recall that $r=b/(2b+1)$. First, we have from~\eqref{eq:finitegraphenergy} for any function $f \in \text{dom}\,\mathcal{E}$ that
\begin{align*}
    r\mathcal{E}(f) \geq r\mathcal{E}_1(f) &\geq \abs{f(p) - f(q)}^2 + b\abs{f(q) - f(z)}^2  \\
    &\geq \left(\dfrac{b}{b+1}\right)^2 \abs{f(p) - f(z)}^2 + b\left(\dfrac{1}{b+1}\right)^2 \abs{f(p) - f(z)}^2 \\
    &\geq \left(1 + \dfrac{1}{b}\right)^{-1} \abs{f(p) - f(z)}^2
\end{align*}
by taking the energy-minimizing value of $f(q) = 1/(b+1)f(p) + b/(b+1)f(z)$. Then, we will consider the function $f_p \in \text{dom}\, \mathcal{E}$ defined by extending
\begin{align*}
    f_p|_{V_1}(q_1) = 1, \quad f_p|_{V_1}(p_1) = f_p|_{V_1}(p_2) = f_p|_{V_1}(q_2) = 0
\end{align*}
to $V_*$ by the energy-minimizing algorithm in Proposition~\ref{prop:renormalizationprop}. We compute that $\mathcal{E}(f_p) = \mathcal{E}_1(f_p|_{V_1}) = r^{-1}$ and that $|f_p(p) - f_p(q)|^2 = 1$, and so $R(p, q) \geq r$. Similarly, by extending the function $f_z \in \text{dom}\, \mathcal{E}$ defined on $V_1$ by
\begin{align*}
    f_z|_{V_1}(q_1) = f_z|_{V_1}(p_1) = 0, \quad f_z|_{V_1}(p_2) = f_z|_{V_1}(q_2) = 1,
\end{align*}
we compute that $\mathcal{E}(f_z) = \mathcal{E}_1(f_z|_{V_1}) = br^{-1}$ and that $|f_z(q) - f_z(z)| = 1$, and so $R(q,r) \geq r/b$. Putting this together, we have for any $f \in \text{dom}\, \mathcal{E}$ that
\begin{equation*}
    \abs{f(p) - f(z)}^2 \leq \left(1 + \dfrac{1}{b}\right) r\mathcal{E}(f) \leq \left(R(p, q) + R(q, z)\right)\mathcal{E}(f),
\end{equation*}
and so we have that $R(p, z) \leq R(p, q) + R(q, z)$ with $R$ satisfying the triangle inequality as is desired.
\end{proof}

\end{definition}

We quickly remark that we can show using the above result that $\text{dom}\,\mathcal{E} / (\text{constant functions})$ is a Hilbert space with respect to $\mathcal{E}$. With the above result, we can finally construct the bubble-diamond fractals as the completion of the graph approximation. 

\begin{definition} \label{def:bdf}

The \textit{bubble-diamond fractal $K_b$ with branching parameter $b$} is the $R$-completion of $V_{\ast}$.

\end{definition}

Similar to above, we can equivalently construct the bubble-diamond fractals as the unique compact set $K_b$ that satisfies
\begin{equation}
    K_b = \bigcup_{i \in \{1, 2, \dots, b + 2\}} F_i(K_b).
\end{equation}
We will now extend the analytic tools from above to $K_b$. First, we will inductively define a \textit{self-similar measure} $\mu$ on $K_b$ given by $\mu(K_b) = 1$ and
\begin{equation}
    \mu(A) = \left(\dfrac{1}{b+2}\right) \sum_{i \in \{1, 2, \dots, b+2\}} \mu(F^{-1}_i A).
\end{equation}
This scaling condition uniquely defines $\mu$ on $K_b$ \cite[Section 1.2]{BobBook2006}. With this measure, we can define the Laplacian on $K_b$.

\begin{definition}\label{def:Lap} Let $u \in \text{dom}\,\mathcal{E}$ and $f$ a continuous on $K_b$. We say that $u \in \text{dom}\,{\Delta}$ with $\Delta u=f$ if 
\begin{equation}
    \mathcal{E}(u, g) = - \int_{K_b}  f gd\mu
\end{equation}
for all $g \in \text{dom}\,\mathcal{E}$ with boundary conditions $g(q_1) = g(q_2) = 0$.
\end{definition}

The connection between our formulation of the Laplacian $\Delta$ on $K_b$ and the graph Laplacian $\Delta_\ell$ on $V_\ell$ as defined in~\eqref{eq:graphlaplacianfinite} is made clear in the following \textit{pointwise formula}.

\begin{proposition}\label{prop:laplacianscaling}

For any $f \in \text{dom}\, \Delta$ and $p \in V_* \setminus V_0$, we have that
\begin{equation}\label{eq:lappt}
    \Delta f(p)= \left(\dfrac{2}{b+1}\right)\lim_{m\to\infty} \left(\dfrac{r}{b+2}\right)^{-m}\Delta_m f(p).
\end{equation}
Conversely, suppose $f$ is continuous on $K_b$ and the right side of~\eqref{eq:lappt} converges uniformly to a continuous function on $V_*\setminus V_0$. Then, $u\in \text{dom}\, \Delta$ and~\eqref{eq:lappt} holds. 

Furthermore, the Laplacian satisfies the following scaling identity
\begin{equation}\label{eq:scalingLap}
\Delta (f \circ F_i) = \left(\dfrac{r}{b+2}\right)  (\Delta f \circ F_i).
\end{equation}

\end{proposition}

\begin{proof}

By \cite[Theorem 2.2.1]{BobBook2006}, it suffices to compute the integral over $K_b$ of the piecewise harmonic function $\psi_x^{(m)}$ on $V_m$ defined by $\psi_x^{(m)}(p) = \delta_{xp}$. First, we note that each $p \in V_m$ is an element of $(b+1)$ many \textit{$m$-cells} $F(K_b) = (F_{i_1} \circ \dots \circ F_{i_m})(K_b)$. On each $m$-cell $F(K_b)$ with boundary points $\{F(q_1), F(q_2)\}$, we note, since there are $(b+2)^m$ many $m$-cells in $K_b$, that
\begin{align*}
    \sum_{q \in V_*} \int_{K_b} \psi_q^{(m)}d\mu = \int_{K_b}d\mu = 1 &\Rightarrow \int_{F(K_b)} \left(\psi_{F(q_1)}^{(m)} + \psi_{F(q_2)}^{(m)}\right)d\mu = \dfrac{1}{(b+2)^m} \\
    &\Rightarrow \int_{F(K_b)} \psi^{(m)}_x dm = \dfrac{1}{2(b+2)^m}.
\end{align*}
By summing the above integral over the $(b+1)$ many $m$-cells containing $x \in V_m$ to compute the integral over $K_b$, we recover the above formula as is desired. The scaling identity is then an immediate consequence of the pointwise formula.
\end{proof}

Now that we have defined $\Delta$ on $K_b$, we can return to establish the correspondence between $\Delta$ and $\mathcal{E}$ that was alluded to in the above subsection. To do so, we will first define the normal derivative of a function $f$ on $K_b$ at a boundary point $q \in V_0$.

\begin{definition}

The \textit{normal derivative} of a function $f$ on $K_b$ at a boundary point $q \in V_0$ is given by
\begin{align}
\partial_n f(q) = \lim_{m \to \infty} r^{-m} \big( f ( q_i) - f\big(x_m)\big),
\end{align}
where $x_m = (F_{b + 1} \circ \dots \circ F_{b+1})(q_2)$ if $q = q_1 \in V_0$ and $x_m = (F_{b + 2} \circ \dots \circ F_{b+2})(q_1)$ if $q = q_2 \in V_0$, both with $m$ applications of $F_i$. In particular, the normal derivative exists for all $u \in \text{dom}\, \Delta$ \cite[Theorem 2.3.2]{BobBook2006}.

\end{definition}

Finally, with these definition of the graph energy, Laplacian, and normal derivative on $K_b$, we can finally state the Gauss-Green formula that will be essential in our development of a theory of polynomials on $K_b$.

\begin{proposition} \textup{\cite[Theorem 2.3.2]{BobBook2006}} For any $u \in \text{dom}\, \Delta$, we have that
\begin{equation}
    \mathcal{E}(u,g) = - \int_{K_b} (\Delta u)g d \mu + \sum_{V_0} g( \partial_n u)
\end{equation}
holds for all $v \in \text{dom}\, \mathcal{E}$.
\end{proposition}

The last analytic tool that we will develop on $K_b$ is finishing our construction of Green's formula. We will extend the construction from~\eqref{eq:greensfunctionfinite} of Green's formula on $V_\ell \times V_\ell$ to $K_b \times K_b$ in the following way.

\begin{definition}\label{def:greensfunction}

We define \textit{Green's function} on $V_* \times V_*$ to be
\begin{equation}
    G(p,q) = \lim_{\ell \to \infty} G_\ell(p,q),
\end{equation}
which we then continuously extend to be a function on $K_b \times K_b$.

\end{definition}

This construction of Green's function will allow us to solve the \textit{Dirichlet problem} $\Delta u = f$ with the boundary conditions $u(q_1) = u(q_2) = 0$ for any continuous function $f$ on $K_b$, or, equivalently, to show that the restriction of the Laplacian $\Delta$ to the domain $u \in \text{dom}\, \Delta$ with $u(q_1) = u(q_2) = 0$ is invertible.

\begin{proposition} \textup{\cite[Theorem 2.6.1]{BobBook2006}} The Dirichlet problem $\Delta u = f$ with $u(q_1) = u(q_2) = 0$ has a unique solution $u \in \text{dom}\, \Delta$ for any continuous function $f$ on $K_b$ given by 
\begin{equation}
    u(p) = -\int_{K_b} G(p, q)f(q) d\mu(q).
\end{equation}

\end{proposition}

One important observation that we will make here is that, when $b=1$, the bubble-diamond fractal $K_b$ reduces to the unit interval. In this case, we have further that the Laplacian, normal derivative, and Green's function respectively reduce to the second derivative, first derivative, and Green's function from classical calculus. This will continue for the rest of the paper, where our construction of the monomials and the Legendre polynomials will reduce down similarly to the classical case.

\section{Polynomials on the Bubble-Diamond Fractals}\label{sec:Poly}

We now define and investigate the class of polynomials on $K_b$. As motivation, we first define and investigate the properties of harmonic functions, i.e., polynomials of degree zero as are defined below. To construct these harmonic functions, we propose an extension algorithm. Subsequently, we use an approach developed in \cite{splinesBob} to construct higher-order polynomials. In particular, as seen below, a \textit{polynomial} on $K_b$ can be viewed as a multiharmonic function.

\begin{definition}\label{def:poly} 

For each $j \geq 0$, we let  $\mathcal{H}_{j} =  \{f: \ \Delta^{j+1}f=0\}$ be the space of all \textit{polynomials} of degree at most $j$. 
\end{definition}

The next result proves that 
$\mathcal{H}_{j}$ is a vector space of dimension $2j+2$. More specifically, the following statements motivated by \cite[Lemma 2.2, Lemma 2.3]{splinesBob} 
hold.

\begin{proposition}\label{prop:polyexpan} For any $j\geq 0$ and $k=1$ or $2$, let $f_{jk}$ be the solution to $\Delta^{j+1}f_{jk}=0$ that satisfies the boundary conditions 
$$\Delta^m f_{j k}(q_n) = \delta_{jm} \delta_{kn}\quad \text{for} \quad n=1,2. $$

\begin{enumerate}
    \item[(a)] For each $j\geq 0$, the set $\{f_{m k}: m=0, 1, \hdots, j; \,  k=1,2\}$ is a basis for $\mathcal{H}_j$. Furthermore,  $f\in \mathcal{H}_j$ if and only if 
\begin{equation}\label{basisexpan}
f=\sum_{m=0}^j\sum_{k=1}^2 (\Delta^m f(q_k))f_{mk}.
\end{equation}
\item[(b)] For $i\in \{1,2, \hdots, b+2\}$ and $k\in \{1,2\}$, we have that
\begin{equation}\label{eq:}
f_{jk}\circ F_i=\sum_{m=0}^j\sum_{n=1}^2\left(\dfrac{r}{b+2}\right)^m f_{(j-m)k}(F_i q_{n})f_{mn}.
\end{equation}
\end{enumerate}
\end{proposition}

\begin{proof} The first part follows by observing that both sides of~\eqref{basisexpan} belong to $\mathcal{H}_j$. The second part is an application of \eqref{eq:scalingLap}.
\end{proof}

The scaling identity~\eqref{eq:} will be particularly important in constructing the multiharmonic functions $f_{jk}$, as it will suffice to compute the values of $f_{jk}(F_iq_n)$ to construct the multiharmonic functions on all of $V_*$.

\subsection{Harmonic Functions}

To motivate the construction of polynomials on $K_b$,  we first describe the basis for the space $\mathcal{H}_0$ of harmonic functions by explicitly computing the harmonic functions $f_{01},f_{02}$, which we will do in the sequel. 
 
 We recall from Proposition \ref{prop:renormalizationprop}, given a function $f$ on $V_0$, that the energy minimizing extension $\widetilde{f}$ of $f$ to $V_1$ is given by
\begin{equation*}
    \begin{cases}
        \widetilde{f}(F_{b+1}q_2) = \left(\dfrac{b+1}{2b+1}\right) f(q_1) + \left(\dfrac{b}{2b+1}\right) f(q_2), \\[3ex]
        \widetilde{f}(F_{b+2}q_1) = \left(\dfrac{b}{2b+1}\right) f(q_1) + \left(\dfrac{b+1}{2b+1}\right) f(q_2).
    \end{cases}
\end{equation*}
From this observation, we iteratively compute the values of the harmonic function $f_{01}$ on $V_*$ by letting $f_{01}(q_1) = 1$ and $f_{01}(q_2) = 0$ and then defining
\begin{align}
\label{eq:harmonicExtension}
\begin{pmatrix}
f_{01}(F_iq_1) \\[1ex]
f_{01}(F_iq_2)
\end{pmatrix}
=
A_i
\begin{pmatrix}
f_{01}(q_1) \\[1ex]
f_{01}(q_2) 
\end{pmatrix},
\end{align}
where $A_i$ are the \textit{harmonic extension matrices}
\begin{align*}
    &A_{b+1} = \begin{pmatrix}
        1 & 0 \\[2ex]
        \dfrac{b+1}{2b+1} & \dfrac{b}{2b+1}
    \end{pmatrix},  \quad 
    A_{b+2} = \begin{pmatrix}
        \dfrac{b}{2b+1} & \dfrac{b+1}{2b+1} \\[2ex]
        0 & 1
    \end{pmatrix},\quad A_{i} = \begin{pmatrix}
        \dfrac{b+1}{2b+1} & \dfrac{b}{2b+1} \\[2ex]
        \dfrac{b}{2b+1} & \dfrac{b+1}{2b+1}
    \end{pmatrix}
\end{align*}
for $i \in \{1, 2, \dots, b\}$. We repeat this \textit{harmonic extension algorithm} iteratively to compute the values of $f_{01}$ on the rest of $V_*$, and we perform a similar computation to construct $f_{02}$. 

We illustrate several harmonic functions  $f_{01}$ for different branching parameters, see Figure~\ref{fig:harmonicf00}. Of particular interest, note that the harmonic extension algorithm generates $f_{01}$ as a linear function on the real line when $b=1$.

\begin{figure}[b!]
    \centering
    \includegraphics[width=0.9\linewidth]{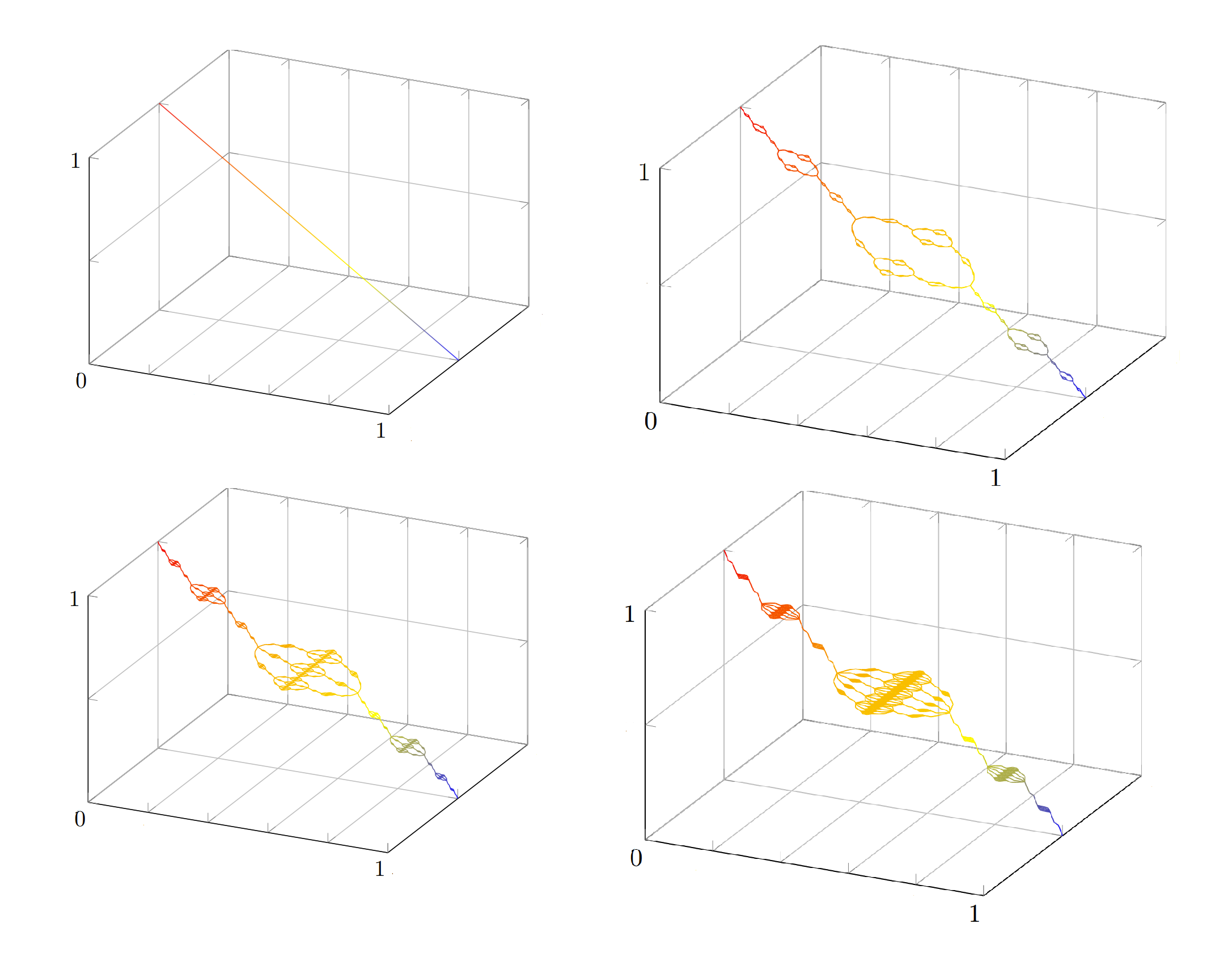}
    \caption{The harmonic function $f_{01}$ on the bubble-diamond fractals with branching parameter $b=1,2,3,5$.}
    \label{fig:harmonicf00}
\end{figure}

\subsection{Higher Order Polynomials} In this section, for all $j\geq 0$,  we evaluate $f_{jk}(F_iq_n)$ for all $i \in \{1, 2, \dots, b+ 2\}$ and  $k \in \{1, 2\}$. These values together with~\eqref{eq:} are needed  to construct the multiharmonic basis $\{f_{01}, f_{02}, f_{11}, f_{12}, \hdots f_{j1},f_{j2}\}$ for $\mathcal{H}_j$. Motivated by the results proved in  \cite{splinesBob}, we need to evaluate the  inner products of the basis elements, i.e., 
\begin{equation}\label{eqforI}
I(jk,j'k')= \int_{K_b} f_{jk} f_{j'k'} d \mu.
\end{equation} 
for all $j,j' \geq 0$ and $k,k' \in \{1, 2\}$. Subsequently, we provide two recursive formulas for $f_{jk}(F_iq_n)$ and $I(jk,j'k')$ that can be simultaneously solved. In particular, we will define the quantities to be solved for as
\begin{align}
\label{eq:Defabpq}
\begin{cases}
a_\ell = I(\ell k,0k), &\\[2ex]
b_\ell = I(\ell k,0n)  &\text{for} \ k \neq n, \\[2ex]
p_\ell = \left(\dfrac{r}{b+2}\right)^{\ell} f_{\ell k}(F_iq_k) ) = \left(\dfrac{r}{b+2} \right)^{\ell} f_{\ell k}(F_kq_i ) &\text{for} \ i \neq k, \\[2ex]
q_\ell = \left(\dfrac{r}{b+2}\right)^{\ell} f_{\ell k}(F_iq_n ) &\text{for} \ i,k, n \ \text{distinct}.
\end{cases}
\end{align}

We begin by evaluating the initial value of these quantities. We note by the self-similarity of $\mu$ that
\begin{equation*}
    I(0k,0k')=\int_{K_b} f_{0k}f_{0k'}d\mu=\frac{1}{b+2}\sum_{i=1}^{b+2}  \int_{K_b} (f_{0k}\circ F_i) (f_{0k'}\circ F_i) d\mu.
\end{equation*}
It follows  from~\eqref{eq:} that the following equation holds
\begin{equation}
\label{eq:recursiceIP}
\begin{pmatrix}
I(01,01) \\
I(01,02) \\
I(02,01) \\
I(02,02) 
\end{pmatrix}
=
\begin{pmatrix}
A(11,11) & A(11,12) & A(11,21) & A(11,22) \\
A(12,11) & A(12,12) & A(12,21) & A(12,22) \\
A(21,11) & A(21,12) & A(21,21) & A(21,22) \\
A(22,11) & A(22,12) & A(22,21) & A(22,22) 
\end{pmatrix}
\begin{pmatrix}
I(01,01) \\
I(01,02) \\
I(02,01) \\
I(02,02) 
\end{pmatrix},
\end{equation}
where the values of $A(kk',nn')$ are computed as
\begin{equation*}
A(kk',nn') = \dfrac{1}{b+2}\sum_{i=1}^{b+2} f_{0k}(F_i(q_n))f_{0k'}(F_i(q_{n'}))
\end{equation*}
from the harmonic extension algorithm.

We observe by symmetry that $a_0 = I(01,01)=I(02,02)$ and $b_0 =I(01,02) = I(02,01)$ and that 
\begin{equation*}
\sum_{i=1}^{2} \sum_{j=1}^{2} I(0i,0j) = \int_{K_b} f_{01}(f_{01} + f_{02}) d \mu + \int_{K_b} f_{02}(f_{01} + f_{02}) d \mu 
= \int_{K_b}( f_{01} + f_{02} ) d \mu = 1.
\end{equation*}
Consequently, we have that \eqref{eq:recursiceIP} reduces to
\begin{equation*}
\begin{pmatrix}
a_0 \\
b_0 
\end{pmatrix}
=
\begin{pmatrix}
A(11,11) + A(11,22) & A(11,12) + A(11,21)  \\
A(12,11)+ A(12,22) & A(12,12) + A(12,21)   
\end{pmatrix}
\begin{pmatrix}
a_0 \\
b_0 
\end{pmatrix}.
\end{equation*}
With this new matrix representation, we can now directly compute the initial values of $a_\ell, b_\ell$ as 
\begin{align}
\label{eqLinitiala0b0}
a_0 = \frac{b+1}{2+4b}, \quad \quad b_0 = \frac{b}{2+4b},
\end{align}
since $a_0 + b_0 = 1/2$ and $\begin{pmatrix} a_0 & b_0 \end{pmatrix}^\top$ is an eigenvector of the above matrix with eigenvalue $1$. Furthermore, we compute the initial values of $p_\ell, q_\ell$ as
\begin{align}\label{eqLinitialp0q0}
p_0 = \frac{b+1}{2b+1}, \quad \quad   q_0 = \frac{b}{2b+1}
\end{align}
from the construction of the harmonic functions above and~\eqref{eq:}.

We will now state our two recursive formulas for computing $a_\ell, b_\ell, p_\ell$, and $q_\ell$. The first recursive formula for $a_\ell, b_\ell$ is an application of~\cite[Lemma 2.4]{splinesBob}.

\begin{theorem}\label{prop:recinner}
For any $j > 0$, the quantities $a_j$ and $b_j$ in (\ref{eq:Defabpq}) satisfy the recursive formula
\begin{equation} \label{recajbj}
\begin{cases}
 \dfrac{(b+2)^{j+1}(2b+1)^{j+1}}{b^j}a_j &= \nu_1 a_j + \nu_2 b_j + (b+1) \displaystyle\sum_{\ell = 0}^{j-1}(a_\ell + b_\ell)\left((b+1)p_{j - \ell} + b q_{j - \ell}\right), \\[1ex]
    \dfrac{(b+2)^{j+1}(2b + 1)^{j+1}}{b^j}b_j &= \upsilon_1 a_j + \upsilon_2 b_j + (b+1) \displaystyle\sum_{\ell = 0}^{j - 1} (a_j + b_j)(b p_{j-\ell} + (b+1)q_{j-\ell}),  
\end{cases}
\end{equation}
where
\begin{equation*}
    \nu_1 = \dfrac{2b^3 + 8b^2 + 7b + 2}{2b+1}, \quad \nu_2 = \dfrac{2b^3 + 6b^2 + 6b + 2}{2b+1}, \quad \upsilon_1 = \dfrac{2b^3 + 4b^2 + 2b}{2b+1}, \quad \upsilon_2 = \dfrac{2b^3 + 6b^2 + 3b}{2b+1}
\end{equation*}
and the initial conditions $a_0, b_0, p_0,$ and $q_0$ are given by~\eqref{eqLinitiala0b0} and~\eqref{eqLinitialp0q0}.
\end{theorem}
\begin{proof}
Let $\lambda =1/(b+2)$ and $r=b/(2b+1)$. We compute~\eqref{eqforI} for $j'=0$ to find that
\begin{align*}
    \left(\dfrac{2b+1}{b[b+1]}\right)\left(\dfrac{(r\lambda)^{-j}}{\lambda}\right) a_j &= \sum_{\ell = 0}^{j} \left[a_\ell\left(\left[1 + \dfrac{1}{b}\right] p_{j - \ell} + q_{j - \ell}\right) + b_\ell \left(\left[1 - \dfrac{1}{b+1}\right]p_{j - \ell} + q_{j - \ell}\right)\right] \\ \notag
    &\quad\quad\quad\quad + \left(\dfrac{2b+1}{b[b+1]}\right)\sum_{\ell = 0}^j \left(b_\ell p_{j - \ell}\right) + \left(\dfrac{2b+1}{b[b+1]}\right)a_j + \dfrac{b_j}{b} \\ \notag
    &= \sum_{\ell = 0}^{j} (a_\ell + b_\ell)\left(\left[1 + \dfrac{1}{b}\right]p_{j - \ell} + q_{j - \ell}\right) + \left(\dfrac{2b+1}{b[b+1]}\right)a_j + \dfrac{b_j}{b} \\ 
    &= \sum_{\ell = 0}^{j - 1} (a_\ell + b_\ell)\left(\left[1 + \dfrac{1}{b}\right]p_{j - \ell} + q_{j - \ell}\right) + \dfrac{2b^3 + 8b^2 + 7b + 2}{b(b+1)(2b+1)}a_j + \dfrac{2(b+1)^2}{b(2b+1)} b_j,
\end{align*} 
which after simplifications yield the first equation as is desired. The second equation is obtained through a comparable method. 
\end{proof}

The second recursive formula for $p_\ell, q_\ell$ is an application of~\cite[Lemma 2.6]{splinesBob}.

\begin{theorem}\label{recurvaluesgraphs}
For any $j > 0$, the quantities  $p_j$ and $q_j$ in (\ref{eq:Defabpq}) satisfy the recursive formula
\begin{equation}\label{recpjqj}
    \begin{cases} (2b+1)p_j &= - \displaystyle\sum_{\ell = 0}^{j - 1}\left[p_{j - \ell - 1}((b+1)^2a_\ell + b^2 b_\ell) + b(b+1)q_{j - \ell - 1}(a_\ell + b_\ell)\right] - (b+1)b_{j-1}, \\[1ex]
    (2b+1)q_j &= - \displaystyle\sum_{\ell = 0}^{j - 1}\left[b(b+1)p_{j - \ell - 1}(a_\ell + b_\ell) + q_{j - \ell - 1}((b+1)^2a_\ell + b^2 b_\ell)\right] - bb_{j-1},
    \end{cases} 
\end{equation}
where the initial conditions $a_0, b_0, p_0,$ and $q_0$ are given by~\eqref{eqLinitiala0b0} and~\eqref{eqLinitialp0q0}.
\end{theorem}

By solving \eqref{recajbj} and \eqref{recpjqj} simultaneously, we are able to compute $f_{jk}(F_iq_n)$, and we then apply \eqref{eq:} to compute the values of $f_{jk}$ on $V_*$ so that $f_{jk}$ can be continuously extended to functions on $K_b$. This completes our construction of a first basis for polynomials on the bubble-diamond fractals. 

 Figure \ref{fig:harmonicfjk} displays the basis function $f_{jk}$ on the bubble-diamond fractal with branching parameter $b = 2$.

\begin{figure}[b!]
    \centering
    \includegraphics[width=0.9\linewidth]{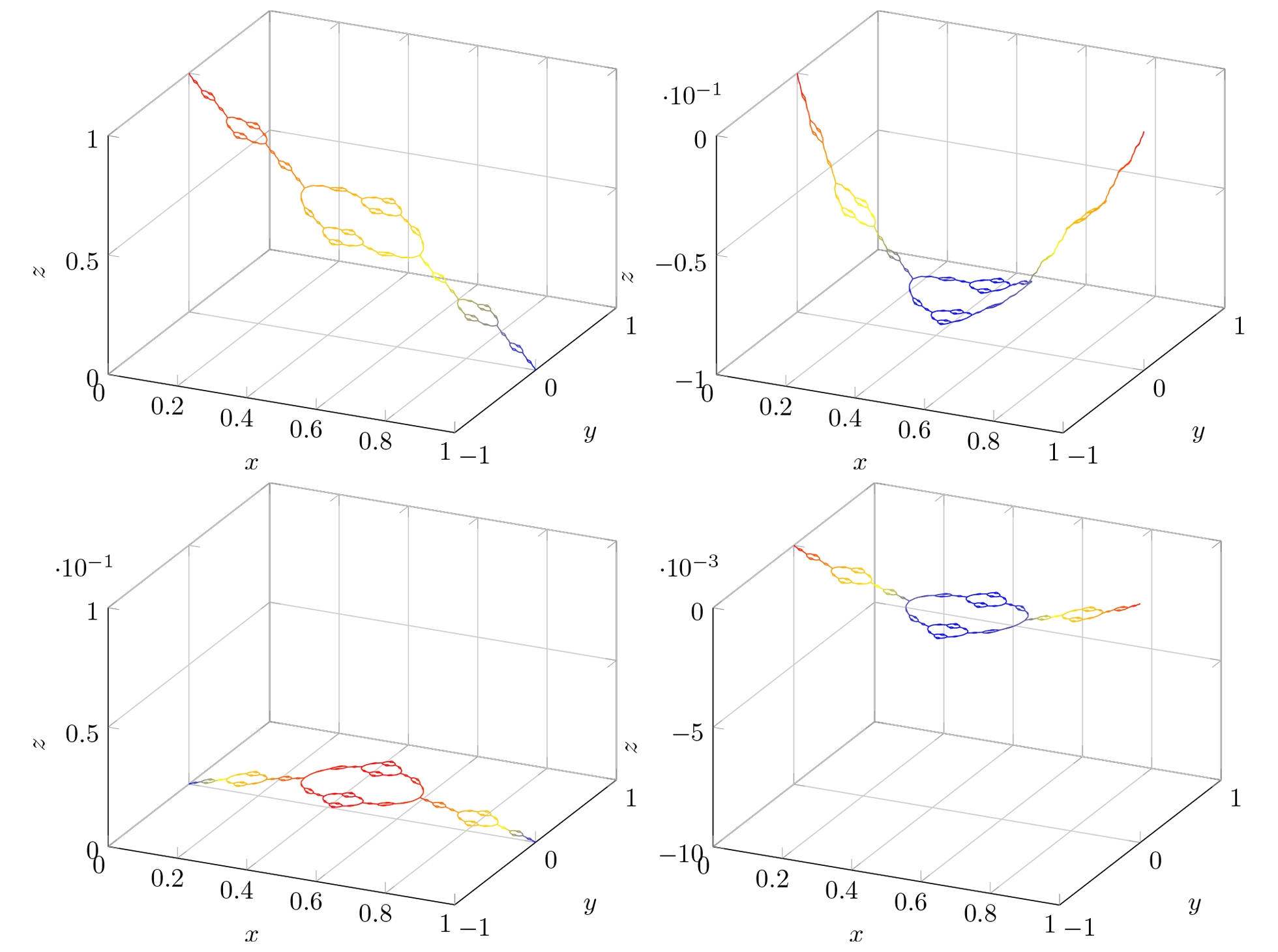}
    \caption{The harmonic functions $f_{01}, f_{11}, f_{21}, f_{31}$ on the bubble-diamond fractal with branching parameter $b=2$.}
    \label{fig:harmonicfjk}
\end{figure}

\section{Orthogonal Polynomials on the Bubble--Diamond Fractals}\label{sec:OP}

Finally, in this section, we construct a sequence of orthogonal polynomials on the bubble-diamond fractals in analogy to the Legendre polynomials. To do so, we consider a different basis for $\mathcal{H}_j$, whose members we call monomials on the bubble-diamond fractals.  When applying the Gram-Schmidt orthogonalization process to these monomials, we obtain the desired sequence of orthogonal polynomials on $K_b$. For comparison to the Sierpinski, we refer to \cite{OST, SOPKasso}. 

We will first define the monomial functions on $K_b$ that we will later construct.

\begin{definition}\label{def:mono} 

For each $j \geq 0$ and $k \in \{1, 2\}$, we will let $P_{jk} \in \mathcal{H}_{j}$ be the \textit{monomials} defined by the boundary conditions
\begin{align*}
    \begin{cases}
    \Delta^m P_{jk}(q_1) = \delta_{mj} \delta_{k1},  \\[1ex]
    \partial_n \Delta^m P_{jk}(q_1) = \delta_{mj} \delta_{k2}
    \end{cases}
\end{align*}
for all $m \in \{0, 1, \dots, j\}$. 

\end{definition}

We observe that the set $\{P_{01}, P_{02}, \dots, P_{j1}, P_{j2}\}$ in $\mathcal{H}_j$ is analogous to the set of monmials  $x^{2j + k}/(2j+k)!$ on the real line. Furthermore, we note that this construction could have similarly been performed by using the boundary conditions at $q_2 \in V_0$. 

As with the multiharmonic functions $f_{jk}$, we have a comparable scaling identity for $P_{jk}$ that follows from Proposition~\ref{prop:recinner}.

\begin{proposition}

For $i \in \{1, 2, \dots, b+2\}$ and $k \in \{1, 2\}$, we have that
\begin{equation}\label{eq:monoscaling}
    P_{jk} \circ F_{i} = \begin{cases}
        \left(\dfrac{r}{b+2}\right)^{j} P_{j1} & k=1, \\
        r\left(\dfrac{r}{b+2}\right)^{j}P_{j2} &k=2.
    \end{cases}
\end{equation}

\end{proposition}

\begin{proof}
Let $\lambda = 1/(b+2)$ and $r = b/(2b+1)$. By taking Laplacians, we compute using ~\eqref{eq:scalingLap} that
\begin{equation*}
    \Delta^j (P_{j1} \circ F_1^m) = (r\lambda)^m P_{01} \circ F_1^m = (r\lambda)^m P_{01}
\end{equation*}
since $P_{01} = 1$ is a constant map. We finally note that $\Delta P_{(j+1)k} = P_{jk}$, which gives us the desired result for $k = 1$. The result for $k=2$ is obtained through a comparable method.
\end{proof}

Similar to our application of Proposition~\ref{prop:polyexpan} above, we can now use~\eqref{eq:monoscaling} to construct $P_{jk}$ on all of $V_*$ from just the value of $P_{jk}(q_2)$, which we compute by solving a pair of recursive formulas. To this end, we will define the quantities
\begin{equation}\label{monoconst}
    \begin{cases}
        \alpha_j = P_{j1}(q_2), \\[1ex]
        \beta_j = P_{j2}(q_2), \\[1ex]
        \eta_j = \partial_n P_{j1}(q_2), \\[1ex]
        \gamma_j = \partial_n P_{j2}(q_2)
    \end{cases}
\end{equation}
with initial values $\alpha_0 = 1$, $\alpha_1 = 1/2$, $\beta_0 = -1$, $\eta_0 = -1$, and $\gamma_0 = 1$ computed by noting that $P_{00} = f_{01} + f_{02}$ and $P_{01} = f_{02}$.

We will now state the two recursive formulas for $\alpha_j, \beta_j$ and $\eta_j, \gamma_j$, which are derived from the same method as \cite[Theorem 2.3, Theorem 2.12]{CalcSGI}.

\begin{theorem}
For any $j > 0$, the quantities $\alpha_j$, $\beta_j$, $\eta_j$, and $\gamma_j$ in \eqref{monoconst} satisfy the recursive formula
\begin{equation}
    \begin{cases}
        \alpha_j &= \zeta_j\cdot \displaystyle\sum_{\ell=1}^{j-1} \alpha_{j-\ell}\left(2\alpha_{\ell} + \sum_{\ell'=1}^\ell\alpha_{\ell-\ell'}\alpha_{\ell'}\right), \\[1ex]
        \beta_j &= \iota_j \cdot \displaystyle\sum_{\ell = 0}^{j - 1}\sum_{\ell' = 0}^{j-\ell} \beta_\ell \alpha_{\ell'}\alpha_{j-\ell-\ell'}, \\[1ex]
        \eta_j &= b_{j-1} + \displaystyle\sum_{\ell=0}^j \alpha_{j-\ell}\alpha_{\ell - 1},\\[1ex]
        \gamma_j &= \displaystyle\sum_{\ell = 0}^j \beta_{j - \ell}\alpha_{\ell - 1},
    \end{cases}
\end{equation}
where
\begin{equation*}
    \zeta_j = \dfrac{(b+1)^2}{(b+2)(2b+1)((b+2)^{j-1}(2b+1)^{j-1}/b^{j-1} -1)},\quad \iota_j = \dfrac{(b+1)^2}{(2b+1)((b+2)^{j}(2b+1)^{j}/b^{j} -1)},
\end{equation*}
and the initial conditions $\alpha_0$, $\alpha_1$, $\beta_0$, $\eta_0$, and $\gamma_0$ are given above.
\end{theorem}

By applying the scaling identity~\eqref{eq:monoscaling}, we once again compute the values of $P_{jk}$ on $V_*$ so that $P_{jk}$ can be continuously extended to functions on $K_b$. This completes our construction of a second basis of monomials on the bubble-diamond fractals. For particular instances of $P_{jk}$ on the bubble-diamond fractal with branching parameter $b = 2$, see Figure \ref{fig:monofjk}.

\begin{figure}[b!]
    \centering
    \includegraphics[width=0.9\linewidth]{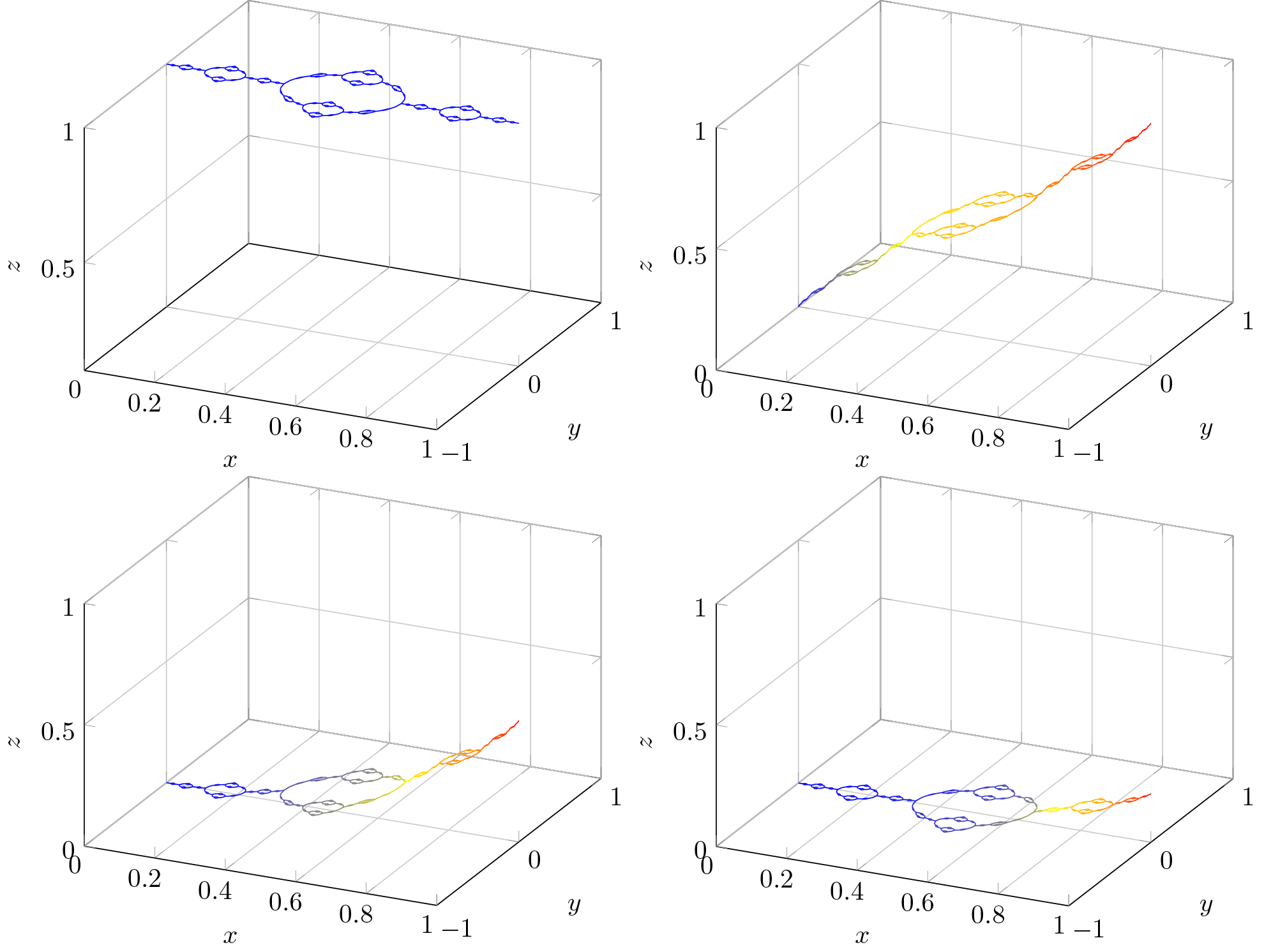}
    \caption{The monomials $P_{01}, P_{02}, P_{11}, P_{12}$ on the bubble-diamond fractal with branching parameter $b=2$.}
    \label{fig:monofjk}
\end{figure}

Finally, we will construct the Legendre polynomials on $K_b$ by performing the Gram-Schmidt orthogonalization process on the above monomial basis. To this end, the next result allows us to compute the inner products for the monomials defined in Definition~\ref{def:mono}

\begin{theorem}
For all $j,j' \geq 0$, we have that
\begin{equation}\label{innerforGS}
\begin{cases} 
    \langle P_{j1}, P_{j'1} \rangle &= \sum_{\ell = 0}^j (\alpha_{j-\ell} \eta_{j' +\ell + 1} - \alpha_{j' + \ell + 1}\eta_{j - \ell}), \\[1ex]
    \langle P_{j1}, P_{j'2} \rangle &= \sum_{\ell = 0}^j (\alpha_{j-\ell} \gamma_{j' +\ell + 1} - \beta_{j' + \ell + 1}\eta_{j - \ell}), \\[1ex]
    \langle P_{j2}, P_{j'2} \rangle &= \sum_{\ell = 0}^j (\beta_{j-\ell} \gamma_{j' +\ell + 1} - \beta_{j' + \ell + 1}\gamma_{j - \ell}).
    \end{cases}
\end{equation}

\end{theorem}

\begin{proof}

Following the procedure of \cite[Lemma 2.1]{OST}, we find that
\begin{equation*}
    \langle P_{ji}, P_{ki'}\rangle = \sum_{\ell = 0}^j \sum_{n=1}^2 \left(P_{(j-\ell)i}(q_n)\partial_n P_{(k+\ell+1)i'}(q_n) - P_{(k + \ell + 1)i'}(q_n)\partial_nP_{(j-\ell)i}(q_n)\right).
\end{equation*}
From the above definitions, this simplifies to
\begin{equation*}
    \langle P_{j1}, P_{k1}\rangle = \sum_{\ell = 0}^j (\alpha_{j-\ell} \eta_{k +\ell + 1} - \alpha_{k + \ell + 1}\eta_{j - \ell}),
\end{equation*}
where we note that $P_{(k+\ell+1)i'}(q_1) = \partial_n P_{(k+\ell+1)i'}(q_1) = 0$ for all $k,\ell \geq 0$. The rest of the identities hold through a comparable method.
\end{proof}

We quickly note that we can alternatively express
\begin{equation*}
    \langle P_{j2}, P_{j'1} \rangle = \sum_{\ell = 0}^j (\beta_{j-\ell} \eta_{j' +\ell + 1} - \alpha_{j' + \ell + 1}\gamma_{j - \ell})
\end{equation*}
by examining the symmetry of the coefficients. 

We now apply the Gram-Schmidt orthogonalization process to the monomials $\{P_{j1}, P_{j2}\}_{j\geq 0}$.  To make more clear the association with the Legendre polynomials from classical calculus, we will make the change in notation $P_{(2j + k)} = P_{jk}$ which makes $P_{(2j + k)} \in \mathcal{H}_j$. We will then construct a family of orthogonal polynomials as follows.

\begin{definition}\label{def:opbdf} The \textit{orthogonal polynomials}  $\{p_j\}_{j\geq 0}$ are obtained by applying the Gram-Schmidt process to $\{P_{j}\}_{j\geq 1}$, where $p_{0} = P_{0}$ and for $j\geq 1$ we have that
\begin{equation}
    p_{j}(x) = P_{j}(x) - \sum_{\ell = 0}^{j-1}d^2_{\ell} \langle P_{j}, p_{\ell}\rangle p_{\ell}(x)
\end{equation}
for coefficients $d_{j} = ||p_{j}||^{-1}$. 
\end{definition}

By normalizing this family of polynomials as $\pi_{j} = d_{j}p_{j}$, we obtain the analog for the Legendre orthogonal polynomials on $K_b$.

\begin{definition}\label{def:ONP}

For each $j \geq 0$, we will call $\pi_{j} \in \mathcal{H}_{j}$ the \textit{Legendre polynomial of degree} $j$ as constructed above, which satisfies 
\begin{equation}
    \langle \pi_{j}, \pi_{j'}\rangle = \delta_{jj'}
\end{equation}
for all $j,j' \geq 0$.

\end{definition}

Figure \ref{fig:legendre} displays some of these orthogonal polynomials on a bubble-diamond fractal with branching parameter $b = 1,2$. In particular, when $b = 1$, we observe that this construction recovers the classical Legendre orthogonal polynomials on the unit interval.

\begin{figure}[b!]
    \centering
    \includegraphics[width=1.0\linewidth]{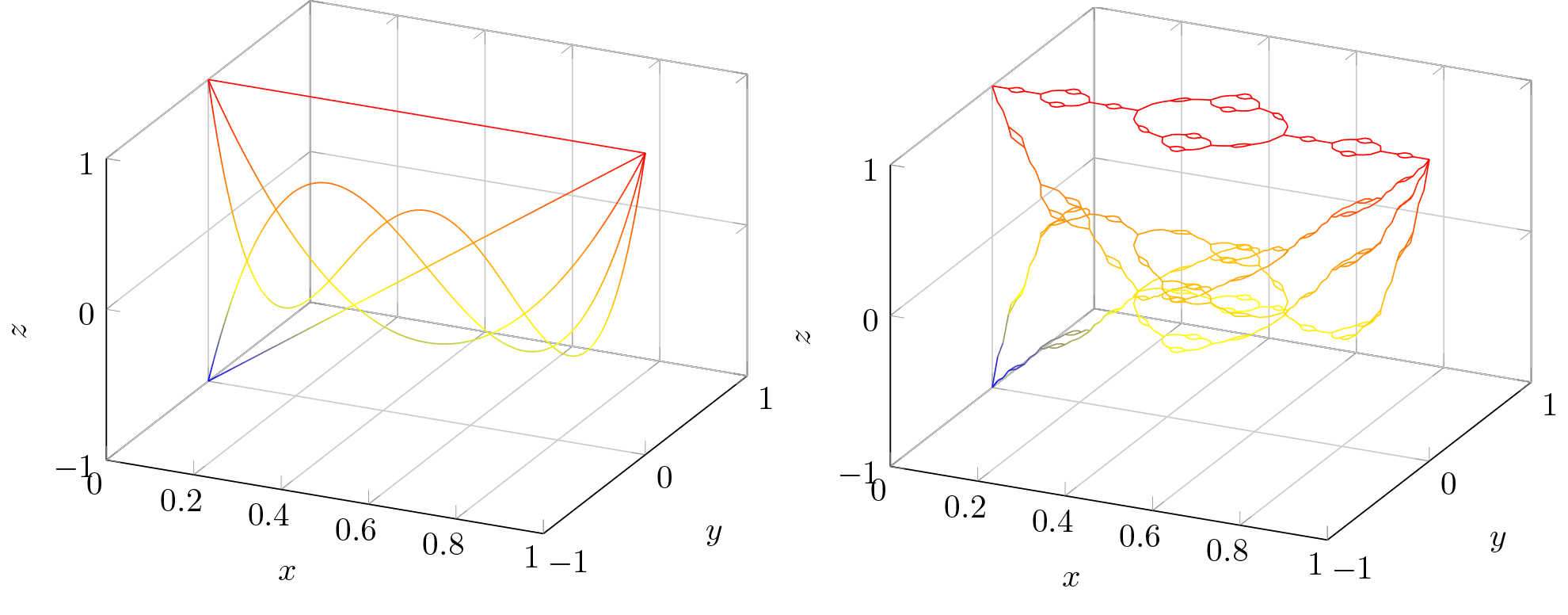}
    \caption{The first five Legendre polynomials on the bubble-diamond fractals with branching parameter $b=1,2$.}
    \label{fig:legendre}
\end{figure}

We show that, like in the classical case,  the family of orthogonal polynomials $\{p_j\}_{j\geq 0}$ satisfies a three-term recursion formula. To this end, consider for all $j \geq 0$ the auxiliary polynomial $g_{j}$ given by 
\begin{equation}
    g_{j+1}(x)=-\int_{K_b}G(x,y)p_{j}(y)\, d\mu(y),
\end{equation}
where $G$ is Green's function defined in Proposition~\ref{def:greensfunction}. It follows from construction that 
\begin{equation}
    \Delta g_{j+1}=p_{j},
\end{equation}
where $g_{j+1}\in \mathcal{H}_{j+1}$ is a polynomial of degree at most $j+1$. As such, we have the following three-term recursion formula.

\begin{theorem}\label{3termsrec} For all $j\geq 0$, we have that
\begin{equation}\label{eq3terms}
g_{j+1}(x)=p_{j+1}(x)+s_{j}p_{j}(x)+t_{j}p_{j-1}(x),
\end{equation} 
where $p_{-1}=g_0=0$. Furthermore, the coefficients  $\{s_{j}\}_{j\geq 0}$ and   $\{t_{j}\}_{j\geq 0}$ are given by 
  \begin{equation}\label{coeffts}
\begin{cases}
    s_{j}=d_{j}^2\langle g_{j+1}, p_{j}\rangle,\\[1ex]
    t_j=d_{j-1}^2d_{j}^{-2} .
\end{cases}  
\end{equation}
In particular, we have that $d_{j}^{-2}=d_{0}^{-2}t_{1}t_{2}\hdots t_{j}$.  
\end{theorem}

\begin{proof} 
For all $j\geq 0$ and $\ell \in \{0, 1, \hdots, j\}$, we see that 
\begin{equation}
    \langle g_{j+1}, p_{\ell } \rangle = - \iint_{K_b\times K_b} G(x,y) p_{j}(y)p_{\ell }(x)d\mu(y) d\mu(x)= \langle p_{j}, g_{\ell+1} \rangle.
\end{equation}
Because $g_{\ell+1}$ is a polynomial of degree at most $\ell+1$, we have that $g_{\ell + 1}$ is orthogonal to all $p_{j}$ for which $\ell+1 < j$. Thus, the expansion of $g_{j+1}\in \mathcal{H}_{j+1}$ with respect to the set of orthogonal polynomials $\{p_{\ell}\}_{\ell=0}^{j+1}$ is of the form 
\begin{equation}
    g_{j+1}=ap_{j+1} +bp_{j}+cp_{j-1}.
\end{equation}
We note that $p_{j}$ can be expressed as a sum of $P_{j}$ and lower-order terms, which implies that $g_{j+1}$ can be expressed as a sum of $P_{j+1}$ and lower order terms. We thus have that  $g_{j+1}-p_{j+1}$ is orthogonal to $p_{j+1}$, and so $a=1$. Taking the inner product of $g_{j+1}$ with $p_{j}$, we find the coefficients $s_{j}$. Furthermore, taking the inner product of $g_{j+1}$ with $p_{j-1}$ yields
\begin{align*}
t_{j}\langle p_{j-1}, p_{j-1}\rangle&=  
\langle g_{j+1}, p_{j-1}\rangle \\
&= \langle p_{j}, g_{j}\rangle\\
&= \langle p_{j}, p_{j}+ s_{j-1}p_{j-1}+t_{j-1}p_{j-2}\rangle \\
&= \langle p_{j}, p_{j}\rangle,
\end{align*}
which gives us the desired formula for the coefficients $t_{j}$. 
\end{proof}

The following result about the orthonormal sequence $\{\pi_j\}_{j\geq 0}$ is a consequence of Theorem~\ref{3termsrec}.

\begin{corollary}\label{3termrecON} Let $\{\pi_j\}_{j=0}^\infty$ be the sequence of orthogonal polynomials given in Definition~\ref{def:ONP}. Let $\tilde{g}_0=0$ and for $j\geq 0$ define $$\tilde{g}_{j+1}(x)=-\int_{K_{b}} G(x,y) \pi_j(y)d\mu(y).$$ The following three-term recursion formula holds.
\begin{equation}\label{eq:3termonp}
\tilde{g}_{j+1}(x)=\sqrt{t_{j+1}} \pi_{j+1}(x)+s_j\pi_j(x)+\sqrt{t_{j}}\pi_{j-1}(x),
\end{equation}
where the sequence $\{s_j\}_{j\geq 0}$ and $\{t_j\}_{j\geq 0}$ are defined in Theorem~\ref{3termsrec}.
    
\end{corollary}

The following result shows that the sequences in the three-term recursion formula~\eqref{eq3terms} are bounded.

\begin{corollary}\label{cor:bdseq}
Let $\{s_j\}_{j\geq 0}$ and $\{t_j\}_{j\geq 0}$ be defined as in Theorem~\ref{3termsrec}. For each $j\geq 0$, we have that
\begin{equation}\label{eq:eqbd}
\begin{cases}
      0\leq t_j\leq \|G\|_2^2, \\[1ex]
  -\|G\|_2\leq s_j\leq 0,
\end{cases}
\end{equation} where $\|G\|_2$ is the $L^2$ norm of the Green function. Furthermore,  
$$ \|p_j\|=d_j^{-1}\leq d_0^{-1}\|G\|_2^{j}. $$
\end{corollary}

\begin{proof}
The fact that $t_j\geq 0$ follows from the second equation in~\eqref{coeffts}. In addition, $$t_jd_{j-1}^{-2}=|\langle g_{j+1}, p_{j-1}\rangle| \leq \|g_{j+1}\|_2 d_{j-1}\leq \|G\|_2 d_{j}^{-1}d_{j-1}^{-1}. $$ Consequently, $t_j\leq \|G\|_2 d_{j-1} d_{j}^{-1}=\|G\|_2 t_j^{1/2}$, from which the first inequality in~\eqref{eq:eqbd} follows. 

The second inequality follows from writing the Green function in terms of the eigenvalues and eigenvectors of the Laplacian, see \cite[Theorem 2.1]{OST}. 
\end{proof}




\section{Numerics}

The figures and numerics included in this paper were generated by an implementation of the above constructions in a \texttt{Python} script by embedding $K_b$ in $[0, 1]^2$, made available as a GitHub repository \href{https://github.com/COsborne25/orthogonal-polynomials-bubble-diamond}{here}. The included \texttt{README.md} file contains more specific details on the implementation of the code.

\section{Acknowledgments}
The authors thank Gamal Mograby for his help during the 2022 VERSEIM-REU at Tufts University. 

E.~Axinn, C.~Osborne, O.~Rigatti, and H.~Shi were supported by the National Science Foundation through the VERSEIM-REU  at Tufts University (DMS-2050412).
K.~A.~Okoudjou was partially supported by the National Science Foundation, grants DMS-1814253, DMS-2205771 and DMS-2309652. 

\bibliographystyle{plain}
\bibliography{refs}

\def\cprime{$'$}
\begin{thebibliography}{10}

\bibitem{Akkermans2009ComplexDim}
E.~Akkermans, G.~Dunne, and A.~Teplyaev.
\newblock Physical consequences of complex dimensions of fractals.
\newblock {\em {EPL} (Europhysics Letters)}, 88(4):40007, nov 2009.

\bibitem{CalcSGII}
Nitsan Ben-Gal, Abby Shaw-Krauss, Robert~S. Strichartz, and Clint Young.
\newblock Calculus on the {S}ierpinski gasket. {II}. {P}oint singularities, eigenfunctions, and normal derivatives of the heat kernel.
\newblock {\em Trans. Amer. Math. Soc.}, 358(9):3883--3936, 2006.

\bibitem{HamblyKumagai2010}
B.~Hambly and T.~Kumagai.
\newblock Diffusion on the scaling limit of the critical percolation cluster in the diamond hierarchical lattice.
\newblock {\em Comm. Math. Phys.}, 295(1):29--69, 2010.

\bibitem{SOPKasso}
Q.~Jiang, T.~Lan, K.~Okoudjou, R.~Strichartz, S.~Sule, S.~Venkat, and Xiaoduo Wang.
\newblock Sobolev orthogonal polynomials on the {S}ierpinski gasket.
\newblock {\em J. Fourier Anal. Appl.}, 27(3):Paper No. 38, 38, 2021.

\bibitem{KigamiBook2001}
J.~Kigami.
\newblock {\em Analysis on fractals}, volume 143 of {\em Cambridge Tracts in Mathematics}.
\newblock Cambridge University Press, Cambridge, 2001.

\bibitem{MalozemovTeplyaev1995}
L.~Malozemov and A.~Teplyaev.
\newblock Pure point spectrum of the {L}aplacians on fractal graphs.
\newblock {\em J. Funct. Anal.}, 129(2):390--405, 1995.

\bibitem{BubbleDiamond2023}
E.~Melville, G.~Mograby, N.~Nagabandi, L.~Rogers, and A.~Teplyaev.
\newblock {Gaps labeling theorem for the Bubble-diamond self-similar graphs}.
\newblock {\em arXiv:2204.11401v2, (2023).}, 2024.

\bibitem{CalcSGI}
J.~Needleman, R.~Strichartz, A.~Teplyaev, and P.~Yung.
\newblock Calculus on the {S}ierpinski gasket. {I}. {P}olynomials, exponentials and power series.
\newblock {\em J. Funct. Anal.}, 215(2):290--340, 2004.

\bibitem{OST}
K.~A. Okoudjou, R.~S. Strichartz, and E.~K. Tuley.
\newblock Orthogonal polynomials on the sierpinski gasket.
\newblock {\em Constructive Approximation}, 37(3):311--340, 2013.

\bibitem{BobBook2006}
R.~Strichartz.
\newblock {\em Differential equations on fractals}.
\newblock Princeton University Press, Princeton, NJ, 2006.
\newblock A tutorial.

\bibitem{splinesBob}
R.~Strichartz and M.~Usher.
\newblock Splines on fractals.
\newblock {\em Math. Proc. Cambridge Philos. Soc.}, 129(2):331--360, 2000.

\end{thebibliography}

\end{document}